\documentclass[12pt, hyp]{amsart}
\usepackage{amsthm}
\usepackage{amsfonts}
\usepackage{amsmath}
\usepackage{mathrsfs}
\usepackage{amssymb}
\usepackage{color}
\usepackage[all]{xy}
\usepackage[pdftex]{graphicx}
\everymath{\displaystyle}
\CompileMatrices

\newtheorem{thm}{Theorem}[section]

\newtheorem{lemma}[thm]{Lemma}
\newtheorem{prop}[thm]{Proposition}

\theoremstyle{definition}

\newtheorem{remark}[thm]{Remark}

\newcommand\R{{\mathbb R}}

\newcommand\C{{\mathbb C}}

\newcommand\N{{\mathbb N}}

\newcommand\Z{{\mathbb Z}}

\newcommand\zero{\textbf{0}}

\newcommand\su{_{\scriptscriptstyle U}}

\newcommand\pr{^{\scriptstyle \R}}
\newcommand\po{^{\scriptscriptstyle O}}
\newcommand\pu{^{\scriptscriptstyle U}}

\newcommand\sr{_{\scriptscriptstyle \R}}
\renewcommand\sc{_{\scriptscriptstyle \C}}

\def\crr{^{\scriptscriptstyle {\it CR}}}

\newcommand\CR{{\mathcal CR}}

\newcommand{\sm}[4]{       \bigl( \begin{smallmatrix} 
					{#1} & {#2} \\ {#3} & {#4}
                   		    \end{smallmatrix} \bigr)         }
\newcommand{\smv}[2]{       \bigl( \begin{smallmatrix} 
					{#1} \\ {#2}
                   		    \end{smallmatrix} \bigr)         }
\newcommand{\smh}[2]{       \bigl( \begin{smallmatrix} 
					{#1} & {#2}
                   		    \end{smallmatrix} \bigr)         }

\newcommand\id{\mathrm{id}}

\newcommand{\diag}{\mathrm{diag}}

\newcommand\snf{ \mathcal{S}}

\everymath{\displaystyle}
\CompileMatrices

\begin{document}

\thispagestyle{empty}

\vspace{-2cm}

\title[$K$-theory for higher rank graph algebras]{Real and complex $K$-theory for higher rank graph algebras arising from cube complexes} 

\begin{abstract} Using the Evans spectral sequence and its counter-part for real $K$-theory, we compute both the real and complex $K$-theory of several infinite families of $C \sp *$-algebras based on higher-rank graphs of rank $3$ and $4$. The higher-rank graphs we consider arise from double-covers of cube complexes. By considering the real and complex $K$-theory together, we are able to carry these computations much further than might be possible considering complex $K$-theory alone. As these algebras are classified by $K$-theory, we are able to characterize the isomorphism classes of the graph algebras in terms of the combinatorial and number-theoretic properties of the construction ingredients.
\end{abstract}

\author{Jeffrey L. Boersema}	
\address{Seattle University \\ Department of Mathematics \\
Seattle, Washington 98133, USA}
\email{boersema@seattle.edu}

\author{Alina Vdovina}
\address{The City College of New York and Graduate Center, CUNY \\ Department of Mathematics \\ New York, 10031, USA}
\email{avdovina@ccny.cuny.edu}

\maketitle

\thispagestyle{empty}


\section{Introduction} 
Higher rank graphs were defined in \cite{KP} and their further theory was developed in \cite{RSY} and \cite{HRSW}.
The main motivation was a systematic study of a large class of $C^*$-algebras classifiable by their $K$-theory.
In spite of the vast literature on the subject, explicit computations of the K-theory of the higher-rank $C^*$-algebras is very rare, especially in rank three and higher.
The first rank three example was done in \cite{Evans}, and the first infinite series of rank three and higher examples were described in \cite{MRV}.
Nevertheless, in both \cite{Evans} and \cite{MRV}, there were open questions on exact order of certain abelian subgroups in K-theory.
In this paper we present several infinite series of $C \sp *$-algebras associated to rank-3 and rank-4 graphs and we compute their K-theory completely and explicitly.  

Not only are we considering the (complex) $C \sp *$-algebras associated to these higher rank graphs, we are also considering real $C \sp *$-algebras for these graphs and we are computing the real $K$-theory. Real $C \sp *$-algebras associated to a higher rank graph, and more generally real $C \sp *$-algebras associated to a higher rank graph with an involution, were introduced in \cite{BG}. The analog of Evans' spectral sequence was also developed in \cite{BG} to compute the real $K$-theory of such algebras. The examples that we consider in this paper are rank-3 and rank-4 graphs with two vertices and a non-trivial involution that swaps the two vertices. We will be calculating the $K$-theory of both real $C \sp *$-algebras: the one associated with the graph with the trivial involution and the one associated with the graph with the non-trivial involution. Previous calculations of $K$-theory for real $C \sp *$-algebras of higher rank graphs have been conducted in \cite{BG} and \cite{BBG}. However, in those cases the graphs either had rank no more than 2, or the graphs could be factored as a product of graphs with rank no more than 2.  Remarkably, we find that the consideration of the real $K$-theory also allows us to compute the (complex) $K$-theory in some cases where that were otherwise intractable. In particular, we are able to resolve the open question in \cite{MRV} mentioned above, and also correct a technical mistake in the description of the $K$-theory results in some cases, in Section~6 of \cite{MRV}.

The complex $C \sp *$-algebras associated to our higher rank graphs fall in the category of purely infinite simple $C \sp *$-algebras classified by $K$-theory in \cite{EK} and \cite{NCP}. Similarly, the real $C \sp *$-algebras in this paper fall in the category of purely infinite simple $C \sp *$-algebras classified in \cite{BRS}. Based on this, we will be able to characterize the isomorphism classes of the resulting algebras in terms of the combinatoric and number-theory properties of the construction ingredients.

The higher-rank graphs that we consider in this paper arise from cube complexes, and their double covers, as in Section 6 of \cite{MRV}. We will review this construction in the next section. We will also review the key preliminary notions, including the definition of the real and complex $C \sp *$-algebras based on higher rank graphs, real $K$-theory, and the spectral sequence technology to calculate the $K$-theory in the real case. 

The geometric core of higher-rank graphs was introduced in \cite{LV}. Initially the higher-rank graphs were defined as small categories in \cite{KP}.
Connections of higher rank graphs with geometry and combinatorics were known before, see, for example, \cite{HRSW}, the combinatorial analogue without reference to category theory was done
only in \cite{LV}. The higher-dimensional digraphs introduced in \cite{LV}  provide a bridge between cube complexes and higher-rank graphs.

The automorphism groups of cube complexes covered by products of trees induce automorphism groups of higher-rank graphs. If the fundamental group of such a cube complex
has a subgroup of index two, then the double-cover of the cube complex has an involution. An involution on the cube complex yields a real structure on the $C \sp *$-algebra associated with to the corresponding higher rank graph.

The topic of $K$-theory for C*-algebras is rich with connections to other topics in mathematics.
For example, in \cite{LSV}, the complex K-theory of C*-algebras is used as an invariant of higher-dimensional Thompson groups which are otherwise very hard to distinguish.  
Also, Matui's HK-conjecture, formulated in \cite{Matui}, proposes an isomorphism between the two complex $K$-groups $K_i(C \sp *_r(\mathcal{G}))$ of an \' etale groupoid $\mathcal{G}$ and the homology groups $\oplus_n H_{2n+i}(\mathcal{G})$. This conjecture has generated quite a bit of interest, and has been shown to hold widely, but not in all generality (see for example, \cite{Deeley} and  \cite{FKPS}). 
The results of this paper indicate that the further development of both real and complex $K$-theory together may shed additional light on these positive connections.

The rest of this paper is organized as follows. Section~\ref{prelim} contains the preliminaries, with a review of rank-$k$ graphs in Section~\ref{p1}  and a review of the construction of the specific family of rank-$k$ graphs that we will investigate in this paper in Section~\ref{p2}. Then in Section~\ref{p3} we discuss the real $C \sp *$-algebra associated to a rank-$k$ graph, and the distinct real $C \sp *$-algebra that one obtains from a rank-$k$ graph with a non-trivial involution. This is followed by a review of $K$-theory in the context of real $C \sp *$-algebras in Section~\ref{p4}.

In subsequent sections, we carry out the $K$-theory calculations for the families of $C \sp *$-algebras under consideration. Specifically, in Section~\ref{3-T} we consider the rank-3 graphs with no involution and in Section~\ref{3-NT} we consider the real $C \sp *$-algebras associated to the same rank-3 graphs with the associated non-trivial involution.  In Section~\ref{4} we consider the rank-4 case, with both the trivial and the non-trivial involution. Section~\ref{NTL} is an appendix containing some number theory results that are needed in the earlier calculations.

\section{Preliminaries} \label{prelim}

\subsection{Higher rank graphs}\label{subs: higher rank graphs} \label{p1}
We recall the definition of a $k$-graph due to Kumjian and Pask \cite{KP}.
For an integer $k \ge 1$, we view $\N^k$ as a monoid under pointwise addition. A $k$-graph is a countable small category $\Lambda$ together with an
assignment of a \emph{degree} $d(\mu)\in\N^k$ to every morphism $\mu\in\Lambda$
such that for all $\mu,\nu,\pi\in \Lambda$ the following hold
\begin{enumerate}
\item $d(\mu\nu)=d(\mu)+d(\nu)$; and
\item whenever $d(\pi)=m+n$ for $m,n \in \N^k$, there is a unique factorisation $\pi=\mu\nu$
    such that $d(\mu)=m$ and $d(\nu)=n$.
\end{enumerate}
Condition~(2) is known  as the \emph{factorisation property} in the $k$-graph. The composition in $\mu\nu$ is understood in the sense of morphisms, thus the source $s(\mu)$ of $\mu$ equals the range $r(\nu)$ of $\nu$. Note that the morphisms of degree $0$ (in $\N^k$) are necessarily the identity morphisms in the category. Denote this set by $\Lambda^0$, and refer to its elements  as \emph{vertices} of $\Lambda$. With $e_1,\dots,e_k$ denoting the generators of $\N^k$, the set  $\Lambda^{e_i}=\{\lambda\in \Lambda\mid d(\lambda)=e_i\}$ consists of edges (or morphisms) of degree $e_i$, for  $i=1,\dots,k$. We write $v\Lambda^n$ for the set of morphisms of degree $n\in \N^k$ with range $v$.

Throughout this paper we are concerned with $k$-graphs where $\Lambda^0$ and all $\Lambda^{e_i}$, $i=1,\dots, k$, are finite. A $k$-graph $\Lambda$ so that $0<\#v\Lambda^n<\infty$ for all $v\in \Lambda^0$ and all $n\in \N^k$ is source free and row-finite. 
The \emph{adjacency matrices} $M_1, \dots, M_k\in\operatorname{Mat}_{\Lambda^0}(\N)$ of $\Lambda$
are $\Lambda^0\times \Lambda^0$ matrices with
\[
M_i(v,w) = |v\Lambda^{e_i}w|.
\]
By the factorisation property, the matrices $M_i$ pairwise commute for $i=1,\dots ,k$. 

\subsection{Rank-$k$ graphs with two vertices} \label{const} \label{p2}

We now review the specific construction of two-vertex $k$-graphs involving cube complexes discussed in Section 6 of \cite{MRV}. 
The construction consists of two steps: 
First, we construct a family of cube complexes with two vertices, covered by products of $k$ trees,
and second, we explain how to get a $k$-graph from each complex. These $k$-graphs happen to have a natural non-trivial involution $\gamma$, which will be important later on. For the background on cube complexes covered by products of $k$ trees, see \cite{LV}, \cite{MRV},  and \cite{RSV}.

\noindent{\bf Step 1.} Let $X_1,...,X_k$ be distinct alphabets such that 
$\left| X_i \right|=m_i$ for $m_i \geq 2$ and $k \geq 1$. Write 
$$X_i=\{x_1^i,x_2^i,...,x_{m_i}^i\}.$$ 
Let $F_i$ be the free group generated by $X_i$.
Then the direct product $$G=F_1\times F_2 \times \ldots \times F_k$$ of $k$ free groups $F_1$,$F_2$,...,$F_k$ has a presentation 
$$G=\langle X_1,X_2,...,X_k \mid [x^i_s,x^j_l]=1, i \neq j=1,...,k; s=1,...,m_i;l=1,...,m_j \rangle,$$ where $[x,y]$ means commutator $xyx^{-1}y^{-1}$.
The group $G$ acts simply transitively on a Cartesian product $\Delta$ of $k$ trees $T_1,T_2,..,T_k$ of valencies $2m_1,2m_2,...,2m_k$ respectively.
The quotient of this action is a cube complex $P$ with one vertex such that the universal cover of $P$ is $\Delta$. The edges of the cube complex $P$ 
are naturally equipped with orientations and labellings by elements of $X=X_1\cup X_2...\cup X_k$ and the $1$-skeleton of $P$ is a wedge of $\sum_{i=1}^k m_i$ circles.
We construct a family of double covers
of $P$ in the following way. A double cover $P^2$ of $P$ has two vertices, say $v_1$ and $v_2$. For each edge $x$ in $P$ there are two edges, say $x_1$ and $x_2$, in the $1$-skeleton of $P^2$. In fact there are two choices for the structure of these edges: either both $x_1$ and $x_2$ are loops, one at $v_1$ and the other at $v_2$; or one of the edges $x_1,x_2$ points from $v_1$ to $v_2$ and the other points from $v_2$ to $v_1$.
We will say that the edge pair $x_1, x_2$ has {\em type one} in the first case, and has {\em type two} otherwise. 

For example, in Figure~1, the edge pair $b_1, b_2$ is type 1 and the edge pair $a_1, a_2$ is type 2. Figures 1,2,3,4,5 show that our double covers are well defined. 

\begin{figure}[!htbp]{Fig. 1}

\includegraphics[height=3cm]{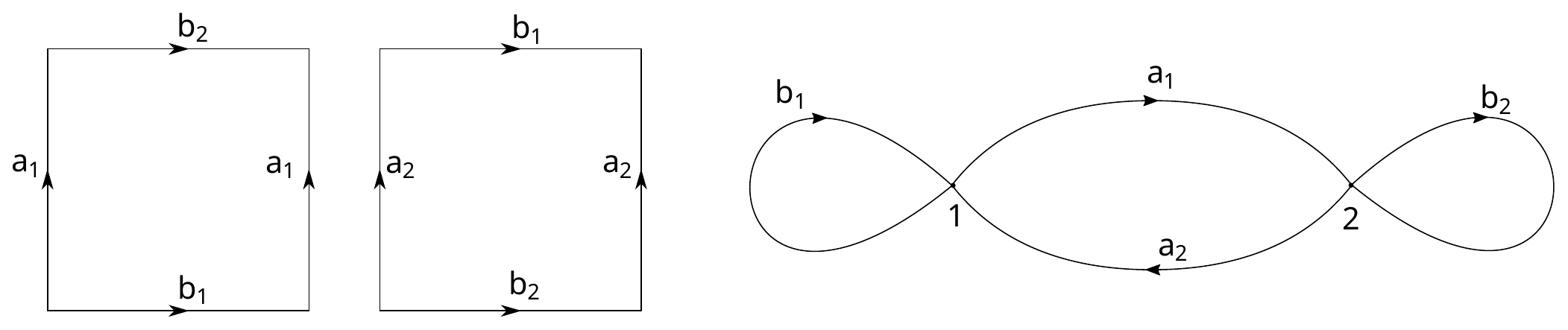}

\end{figure}

In the double covers we consider, we stipulate that at least one edge pair has type two (so the graph is connected) and that all of the edge pairs with labels in the same set $X_i$ will have the same type.

\begin{figure}[!htbp]{Fig. 2}

\includegraphics[height=4cm]{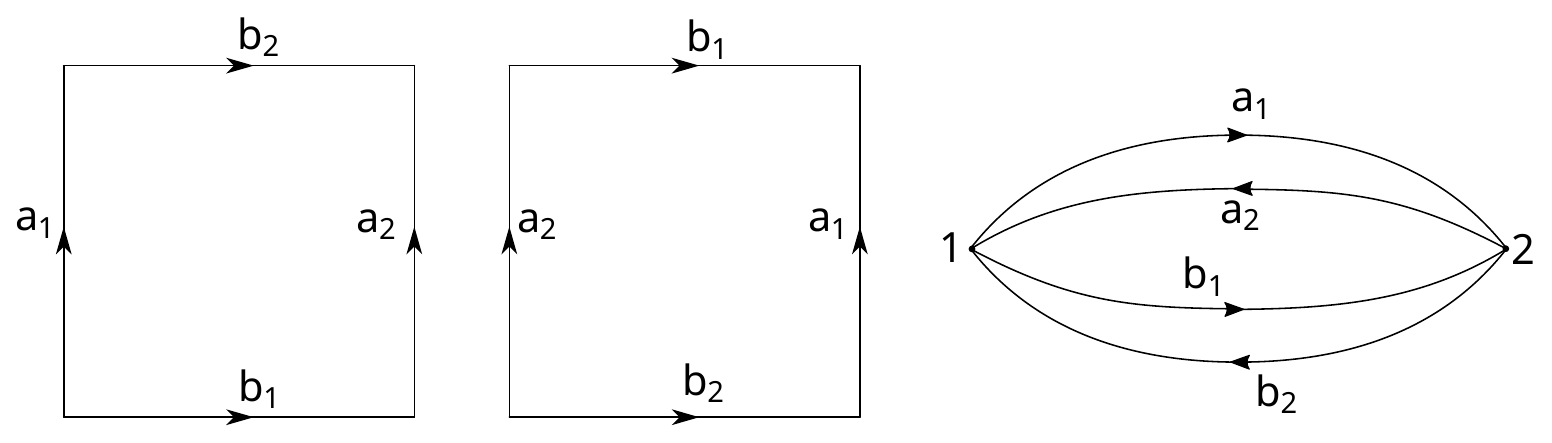}

\end{figure}

\begin{figure}[!htbp]{Fig. 3}

\includegraphics[height=4cm]{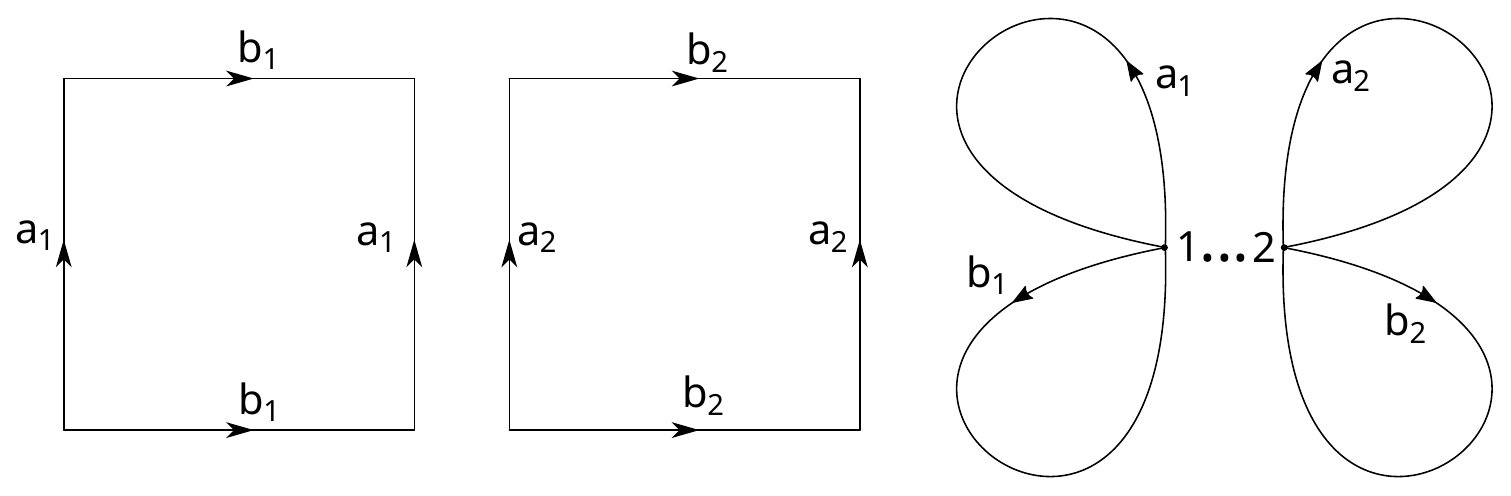}

\end{figure}

\begin{figure}[!htbp]{Fig. 4}

\includegraphics[height=4cm]{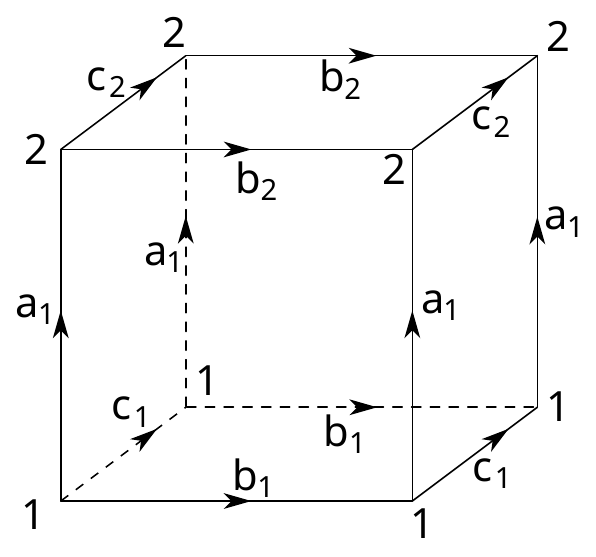}

\end{figure}

\begin{figure}[!htbp]{Fig. 5}

\includegraphics[height=4cm]{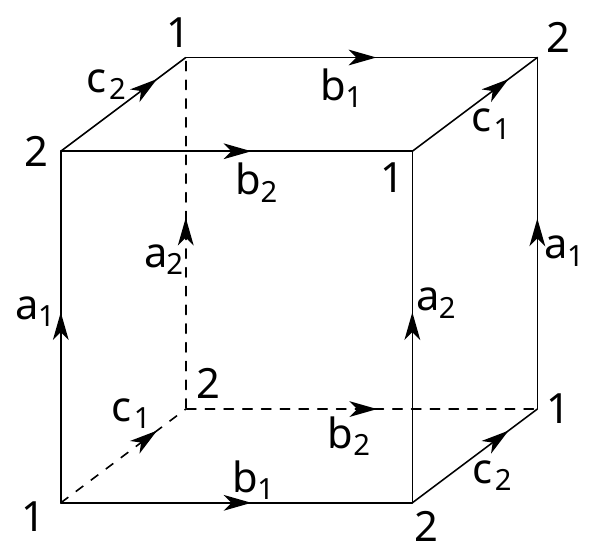}

\end{figure}

\noindent{\bf Step 2.} We explain now how to construct a $k$-graph $C$ from the cube complex $P^2$. The graph $C$ will have the same set of vertices as $P^2$, but the number of edges will double. Specifically, for each edge $x$ in $P^2$, we obtain two edges $x$ and $x'$ where $s(x) = r(x')$ and $s(x') = r(x)$. Furthermore, the degree of $x$ and $x'$ is $e_i$, descending from the labels associated from the edges of $\Delta$ (colloquially we say that the edges $x$ and $x'$ have color $i$).
 Each geometric square $abcd$ in $P^2$ will give rise to four squares (or commutativity relations) in $C$:
 namely, $$ab=d^{\prime}c^{\prime}, bc=a^{\prime}d^{\prime}, cd=b^{\prime}a^{\prime}, da=c^{\prime}b^{\prime}.$$ 
 For example, the square $a_1b_2a_2^{-1}b_1^{-1}$ (the front face of the cube in Fig. 5) will give rise to four squares in $C$: 
 namely, 
 $$a_1b_2=b_1a_2, ~b_2a_2^{\prime}=a_1^{\prime}b_1, ~a_2^{\prime}b_1^{\prime}=b_2^{\prime}a_1^{\prime}, ~b_1^{\prime}a_1=a_2b_2^{\prime}.$$
 
 This completes the description of the construction of a large collection of examples of rank-$k$ graphs.
 In the rest of this paper we will consider graphs that arise from this construction, restricting our attention to the ones in which $P$ has one vertex, so that the rank-$k$ graph $C$ has two vertices (and the number of edges is a multiple of 4).
 
 To proceed, we need to define two special kinds of matrices, 
$$D_i=\begin{bmatrix} 2m_i&0\\ 0&2m_i \end{bmatrix} \text{~and~}
T_i = \begin{bmatrix} 0&2m_i\\ 2m_i&0 \end{bmatrix}
.$$
If the edges in $P$ associated with color $i$ have lifts in $P^2$ that are type 1, then the adjacency matrix of the rank-$k$ graph $C$ will have the form $D_i$ for $i \in \{1,2, \dots, k\}$. 
If the lifts are of type 2, then the  adjacency matrix will have the form $T_i$.

\subsection{Real and Complex $C \sp *$-algebras}  \label{p3}
From these higher-rank graphs, we construct real and complex $C \sp *$-algebras, following \cite{BG} and \cite{KP}, as follows.
For any row-finite source-free k-graph $\Lambda$, $C \sp *(\Lambda)$ is the universal complex $C \sp *$-algebra generated by partial isometries $s_\lambda$, for $\lambda \in \Lambda$, subject to the relations
\begin{enumerate}
\item For each $v \in \Lambda^0$, $s_v$ is a projection, and $s_v s_w = \delta_{v,w} s_v$.
\item For each $\lambda \in \Lambda, \ s_\lambda^* s_\lambda = s_{s(\lambda)}$.
\item For each $\lambda, \mu \in \Lambda, \ s_\lambda s_\mu = s_{\lambda \mu}$.
\item For each $v \in \Lambda^0$ and each $n \in \N^k$, 
$\displaystyle s_v = \sum_{\lambda \in v\Lambda^n} s_\lambda s_\lambda^*.$
\end{enumerate}

The real $C \sp *$-algebra $C \sp * \sr (\Lambda)$ is the universal real $C \sp *$-algebra generated by the same partial isometries $s_\lambda$ as above subject to the same relations. We can and do typically represent $C \sp * \sr(\Lambda)$ as the real subalgebra of $C \sp *(\Lambda)$ generated by $s_\lambda$. 
Thus $C \sp * \sr (\Lambda)$ is the closure of the set of all real linear combinations of products of $s_\lambda$ and $s_\lambda^*$.

In addition, there is a obvious involution $\gamma$ on the graph $\Lambda$ that interchanges the two vertices and interchanges pairs of edges in a way consistent with the action on the vertices. In this situation, there is a different real $C \sp *$-algebra $C\sr \sp * (\Lambda, \gamma)$, associated to this graph with involution, as constructed in \cite{BG}, which is represented as the real $C \sp *$-algebra in $C \sp *(\Lambda)$ generated by the elements of the form $z s_\lambda + \overline{z} s_{\gamma(\lambda)}$ for $z \in \C$.

The two real $C \sp *$-algebras $C \sp * \sr (\Lambda)$ and $C \sp * \sr (\Lambda, \gamma)$ are both real structures associated with $C \sp *(\Lambda)$, in the sense that the complexification of each one is isomorphic to the complex $C \sp *$-algebra  $C \sp * (\Lambda)$. A typical problem in the theory of real $C \sp *$-algebras is to identify up to isomorphism all of the real structures associated with a given complex $C \sp *$-algebra.
The constructions of these real $C \sp *$-algebras depend on the integer values of $m_i$ (for $i \in \{1,2, \dots, k\}$), on the choices of the type of lifts for each $i$ (that is the form of the adjacency matrices $M_i$), and the choice of whether we are considering the real $C \sp *$-algebra $C \sp * \sr(\Lambda)$ or $C \sp * \sr(\Lambda, \gamma)$. 

\subsection{K-theory} \label{p4}

In our work, we will use the abbreviated version of united $K$-theory $K\crr(A)$ that was introduced in \cite{Boersema02} for real $C \sp *$-algebras. From Theorem~10.2 of \cite{BRS}, this invariant classifies the category of real purely infinite simple $C \sp *$-algebras consisting of exactly those real $C \sp *$-aglebras whose complexifications fall under the classification theorem for complex purely infinite simple $C \sp *$-algebras, by Kirchberg and Phillips in \cite{EK} and \cite{NCP}.
This category includes all of the real graph algebras we will consider in this paper.
Specifically, for a real $C \sp *$-algebra $A$ we define
$$
K\crr(A) = \{ KO_* (A), KU_*(A)\} 
$$
where $KO_*(A)$ is the standard 8-periodic real $K$-theory for a real $C \sp *$-algebra and $KU_*(A) = K_*(\C \otimes\sc A)$ is the standard 2-periodic $K$-theory of the complexification of $A$. The invariant $K\crr(A)$ also includes the natural transformations 
\begin{align*}
r_i &\colon KU_i(A) \rightarrow KO_i(A) && \text{induced by the standard inclusion }  \mathbb{C} \rightarrow M_2(\mathbb{R})  \\
c_i & \colon KO_i(A) \rightarrow KU_i(A) && \text{induced by the standard inclusion }  \mathbb{R} \rightarrow \mathbb{C} \\
\psi_i &\colon KU_i(A) \rightarrow KU_i(A) &&  \text{induced by conjugation }  \mathbb{C} \rightarrow \mathbb{C}  \\
\eta_i &\colon KO_i(A) \rightarrow KO_{i+1}(A) &&  \text{induced by multiplication by $\eta \in KO_1(\R) = \Z_2$} \\
 \xi_i &\colon KO_i(A) \rightarrow KO_{i+4}(A) &&  \text{induced by multiplication by $\xi \in KO_4(\R) = \Z$. } 
\end{align*}

The additional structure tends to aid in the computations of $KO_*(A)$ because the natural transformations satisfy the relations
\begin{align*}
&rc=2 && cr=1+\psi && 2\eta =0 \\
&r\psi=r && \psi^2=\id && \eta^3=0 \\
&\psi c=c && \psi\beta\su =-\beta\su\psi && \xi = r\beta\su^2c \\
\end{align*}
and they fit into a long exact sequence
\[
\cdots \xrightarrow{r\beta\su^{-1}} KO_i(A) \xrightarrow{\eta} KO_{i+1}(A) \xrightarrow{c} 
	KU_{i+1}(A) \xrightarrow{r\beta\su^{-1}} KO_{i-1}(A) \xrightarrow{\eta} \cdots 
	\label{eq:LES}
	\]
where $\beta\su$ is the Bott periodicity map on complex $K$-theory.
The target category of the functor $K\crr(-)$ is the category of all $\mathcal{CR}$-modules.

To compute $K\crr(C \sp * \sr(\Lambda, \gamma))$, we will use the spectral sequence of \cite[Theorem 3.13]{BG}, which generalizes the spectral sequence of \cite{Evans} for complex $K$-theory. 

The $E^2$ page of the spectral sequence arises from the homology of a certain chain complex $\mathcal{C}$, which is based on the $\mathcal{CR}$-modules $K\crr(\R)$ and $K\crr(\C)$ and relies on the combinatorial data of $\Lambda$ and $\gamma$. We will review the details of the formation of this spectral sequence in our calculations in the following sections. The spectral sequence converges to $K\crr( C \sp * \sr (\Lambda, \gamma))$ in the sense that there is a filtration of $K\crr( C \sp * \sr (\Lambda, \gamma))$, the subfactors of which appear as the groups of the $E^\infty$ page. Specifically, the groups $KO_n( C \sp * \sr (\Lambda, \gamma))$ and $KU_n( C \sp * \sr(\Lambda, \gamma))$ are obtained from the groups $(E^\infty_{p,q})\po$ and $(E^\infty_{p,q})\pu$where $p + q = n$.

This spectral sequence exists in the category of $\mathcal{CR}$-modules, which means that it has both a real component and a complex component, as alluded to above, and these components are connected by the natural transformations including $r$ and $c$. This is the case on each page of the spectral sequence starting with the chain complex $\mathcal{C}$ and the natural transformations commute with the differentials $d$. The complex component of this spectral sequence coincides with the spectral sequence of Evans in \cite{Evans}.

For reference, the groups of $K\crr(\R)$ and $K\crr(\C)$ are shown below from Tables~1 and 2 of \cite{BG}.

\begin{figure}[!htbp]{Fig. 6 -- $K\crr( \R )$}
$$\begin{array}{|c|c|c|c|c|c|c|c|c|}  
\hline \hline  
n & \makebox[1cm][c]{0} & \makebox[1cm][c]{1} & 
\makebox[1cm][c]{2} & \makebox[1cm][c]{3} 
& \makebox[1cm][c]{4} & \makebox[1cm][c]{5} 
& \makebox[1cm][c]{6} & \makebox[1cm][c]{7}  \\

\hline  \hline
KO_n(\R )
& \Z  & \Z_2   & \Z_2 & 0 
&   \Z   & 0 & 0 & 0  \\
\hline  
KU_n(\R)
& \Z & 0 & \Z & 0 & \Z & 0 & \Z  & 0  \\
\hline \hline
c_n & 1 & 0 & 0 & 0 & 2 & 0 & 0 & 0   \\
\hline
r_n & 2 & 0 & 1 & 0 & 1 & 0 & 0 & 0       \\
\hline
\psi_n & 1 & 0 & -1 & 0 &  1 & 0 & -1 & 0    \\
\hline
\eta_n & 1 & 1 & 0 & 0 & 0 & 0 & 0 & 0     \\
\hline \hline
\end{array}$$
\end{figure}

\begin{figure}[!htbp]{Fig. 7 -- $K\crr( \C )$}
$$\begin{array}{|c|c|c|c|c|c|c|c|c|}  
\hline \hline  
n & \makebox[1cm][c]{0} & \makebox[1cm][c]{1} & 
\makebox[1cm][c]{2} & \makebox[1cm][c]{3} 
& \makebox[1cm][c]{4} & \makebox[1cm][c]{5} 
& \makebox[1cm][c]{6} & \makebox[1cm][c]{7}  \\

\hline  \hline
KO_n(\C))
& \Z  & 0  & \Z & 0 
& \Z  & 0  & \Z & 0  \\
\hline  
KU_n(\C)
& \Z^2 & 0 & \Z^2 & 0 & \Z^2 & 0 & \Z^2 & 0  \\
\hline \hline
c_n & \smv{1}{1} & 0 & \smv{1}{-1} & 0 & \smv{1}{1} & 0 & \smv{1}{-1} & 0   \\
\hline
r_n & \smh{1}{1} & 0 & \smh{1}{-1} & 0 & \smh{1}{1} & 0 & \smh{1}{-1} & 0    \\
\hline
\psi_n & \sm{0}{1}{1}{0} & 0 &  \sm{0}{-1}{-1}{0}  & 0 & \sm{0}{1}{1}{0} & 0 &  \sm{0}{-1}{-1}{0}  & 0   \\
\hline
\eta_n & 0 & 0 & 0 & 0 & 0 & 0 & 0 & 0     \\
\hline \hline
\end{array}$$

\end{figure}


\section{The rank-3 case, with no involution} \label{3-T}

Let $\Lambda$ be a rank-3 graph of the form discussed in Section~\ref{const}.
Specifically, $\Lambda$ is a two-vertex graph and 
the incidence matrices $M_i$ for $\Lambda$ each have the form
$$D_i = \begin{bmatrix} 2m_i & 0 \\ 0 & 2 m_i \end{bmatrix}
	\quad \text{or} \quad
T_i = \begin{bmatrix} 0 & 2 n_i \\ 2n_i & 0 \end{bmatrix} 
$$
for $i = 1,2,3$
(with the restriction that at least one of the incidence matrices must have the form $T_i$).
From \cite{MRV}, we have the complex $K$-theory $KU( C \sr \sp* (\Lambda) = K_*( C \sp * (\Lambda))$ given by
\begin{align*}
K_*( C \sp * (\Lambda )) &= 0 && \text{if $g = 1$} \\
K_0( C \sp * (\Lambda )) &= \text{some extension of $\Z_g$ by $\Z_g$} && \text{if $g \geq 3$} \\
K_1(C \sp *(\Lambda)) &= \Z_g^2 && \text{if $g \geq 3$.}
\end{align*}
However, in \cite{MRV} the nature of the extension was not determined. Furthermore, in the cases where more than one of the matrices $M_i$ has the off-diagonal form, the formula for $g$ in \cite{MRV} is incorrect. In this section, we will compute both $KO_*(C \sr \sp *(\Lambda))$ and $KU_*(C \sr \sp * (\Lambda))$ (thereby determining the previously unknown extension), and we correct the formula for $g$.

Here we define $g$ as follows:
\begin{enumerate}
\item If $M_1 = T_1, M_2 = D_2, M_3 = D_3$ then
$$g = \gcd(1 - 4n_1^2, 1 - 2m_2, 1-2m_3) \; .$$
\item If $M_1 = T_1, M_2 = T_2, M_3 = D_3$ then
$$g = \gcd(1 - 4n_1^2, 1-4n_2^2, 1-4n_1 n_2, 1-2m_3) \; .$$
\item If $M_1 = T_1, M_2 = T_2, M_3 = T_3$ then
$$g = \gcd(1 - 4n_1^2, 1-4n_2^2, 1 - 4n_3^2,1-4n_1 n_2, 1-4n_1 n_3, 1-4n_2 n_3) \; .$$
\end{enumerate}
This formula for $g$ agrees with \cite{MRV} (Proposition 6.2) in case (a), but is a correction in cases (b) and (c).

\begin{prop} \label{calc1}
For a rank-3 graph as described above, $K\crr( C \sr \sp *( \Lambda))$ is given by the table below, for $g \geq 3$.
\end{prop}

Note that if $g = 1$, then $K\crr(C \sr \sp *(\Lambda)) = 0$ in all degrees.

$$K\crr(C \sp * \sr(\Lambda))$$
$$\begin{array}{|c|c|c|c|c|c|c|c|c|}  
\hline \hline  
n & \makebox[1cm][c]{0} & \makebox[1cm][c]{1} & 
\makebox[1cm][c]{2} & \makebox[1cm][c]{3} 
& \makebox[1cm][c]{4} & \makebox[1cm][c]{5} 
& \makebox[1cm][c]{6} & \makebox[1cm][c]{7}  \\

\hline  \hline
KO_n(C \sp * \sr (\Lambda))
& \Z_g  & \Z_g^2   & \Z_g & 0 
& \Z_g  & \Z_g^2   & \Z_g & 0  \\
\hline  
KU_n(C \sp * \sr (\Lambda))
& \Z_g^2 & \Z_g^2 & \Z_g^2 & \Z_g^2 & \Z_g^2 & \Z_g^2 & \Z_g^2 & \Z_g^2    \\
\hline \hline
\end{array}$$

\begin{proof} The graph $\Lambda$ has two vertices. So, following \cite{BG}, we
set $\mathcal{A} = K\crr(\R) \oplus K\crr(\R)$ and consider the chain complex
\[ 0 \rightarrow \mathcal{A} \xrightarrow{\partial_3} \mathcal{A}^3 \xrightarrow{\partial_2} 
				\mathcal{A}^3 \xrightarrow{\partial_1} \mathcal{A} \rightarrow 0 \]
the homology of which gives the $E^2$ page of a spectral sequence which converges to $K\crr( C \sp * (\Lambda))$.

The complex part of this chain complex in degree 0 is exactly the chain complex analyzed in the proof of Proposition 6.3 of \cite{MRV}, specifically we have
\begin{equation}  \label{chaincomplex1}
	0 \rightarrow \Z^2 \xrightarrow{\partial_3} \Z^6 \xrightarrow{\partial_2}
				\Z^6 \xrightarrow{\partial_1} \Z^2 \rightarrow 0 
\end{equation}
where 
\begin{align*}
\partial_1 &= \begin{bmatrix} I - M_1^T & I - M_2^T & I - M_3^T \end{bmatrix} \\
\partial_2 &= \begin{bmatrix} -(I - M_2^T) & -(I - M_3^T) & 0  \\
					I - M_1^T & 0 & -(I - M_3^T) \\
					0 & I - M_1^T & I - M_2^T &  \end{bmatrix} \\
\partial_3 &=\begin{bmatrix} I - M_3^T \\ -(I - M_2^T) \\ I - M_1^T \end{bmatrix}  \; .
\end{align*}
We refer to Lemma~\ref{snf1} at the end of this section the calculation of the Smith normal forms of these matrices, which come out to the following:
\[ 
\snf(\partial_1) = \snf( \partial_3)^T =  \begin{bmatrix} 1 & 0 & 0 & 0 & 0 & 0 \\ 0 & g & 0 & 0 & 0 & 0 \end{bmatrix} \\
\quad \text{and} \quad
\snf(\partial_2) = \diag(1, 1, g, g, 0, 0)  \; .
\]
Thus
the homology of this chain complex is
$H_*(\mathcal{C}) = (\Z_g, \Z_g^2, \Z_g, 0)$ in degrees $p = 0,1,2,3$.

The real part of this chain complex has period 8. In degrees 0 and 4, it is identical to the complex part of the chain complex shown above in (\ref{chaincomplex1})
with the same partial maps, so the homology is the same.
The real part of this chain complex in degrees 1 and 2 consists of $2$-torsion subgroups
\[ 0 \rightarrow \Z_2^2 \xrightarrow{\partial_3} \Z_2^6 \xrightarrow{\partial_2}
				\Z_2^6 \xrightarrow{\partial_1} \Z_2^2 \rightarrow 0 \]
but the matrices describing the partials are the same as above, modulo 2.
Since $g$ is odd, the chain complex is exact and the homology vanishes.

Therefore, the $E^2$ page of the spectral sequence of $K\crr( C \sp * (\Lambda))$ looks like the following, in the real and complex parts.

	\begin{gather*}
			\text{ \underline{ $E^2_{p,q} $ (for $g$ odd)  }}\\
			\def\vvline{\hfil\kern\arraycolsep\vline\kern-\arraycolsep\hfilneg}
			\begin{array}{ ccccc }
				\multicolumn
				{5}{c}{ \underline{ \text{real part}}} \\
				\vspace{.25cm} \\
				\vdots \vvline & \hspace{.3cm} \vdots \hspace{.3cm} & \hspace{.3cm} \vdots \hspace{.3cm} 
						& \hspace{.3cm} \vdots \hspace{.3cm} & \hspace{.1cm} \vdots \hspace{.1cm} \\
				7  \vvline & 0 & 0  & 0 & 0  \\
				6  \vvline & 0 & 0  & 0 & 0 \\
				5   \vvline & 0 & 0  & 0 & 0 \\
				4   \vvline  &  \Z_{g} &  \Z_g^2 & \Z_g & 0  \\
				3   \vvline & 0 & 0  & 0  & 0 \\
				2   \vvline & 0 & 0  & 0  & 0 \\
				1  \vvline & 0 & 0  & 0 & 0 \\
				0  \vvline  &  \Z_{g} &  \Z_g^2 & \Z_g  & 0 \\  \hline
				~ \vvline & 0 & 1 & 2 & 3
			\end{array}
			\hspace{3cm}
			\begin{array}{ ccccc }
				\multicolumn
				{5}{c}{ \underline{ \text{complex part}}} \\
				\vspace{.25cm} \\
				\vdots  \vvline & \hspace{.3cm} \vdots \hspace{.3cm} & \hspace{.3cm} \vdots \hspace{.3cm} 
						& \hspace{.3cm} \vdots \hspace{.3cm} & \hspace{.1cm} \vdots \hspace{.1cm}  \\
				7 \vvline  & 0    &  0 &  0  & 0 \\
				6    \vvline  &  \Z_{g} &  \Z_g^2 & \Z_g & 0  \\
				5  \vvline  & 0    &  0 &  0 & 0   \\
				4   \vvline  &  \Z_{g} &  \Z_g^2 & \Z_g & 0 \\
				3   \vvline  & 0    &  0 &  0  & 0  \\
				2    \vvline  &   \Z_{g} &  \Z_g^2 & \Z_g & 0\\
				1   \vvline  & 0    &  0 &  0  & 0  \\
				0   \vvline  &  \Z_{g} &  \Z_g^2 & \Z_g & 0 \\ \hline
				~ \vvline & 0 & 1 & 2 & 3
			\end{array}
		\end{gather*}

Then the structure of this spectral sequence implies that there are no non-trivial differentials. Therefore $E_{p,q}^2 = E_{p,q}^\infty$ in both the real and complex part. Furthermore, in the real case, there is never more than one non-trivial group along a single diagonal $p + q = i$ (for $i$ fixed), so there are no non-trivial extension problems for $KO_i( C \sp * \sr (\Lambda))$. Thus the real $K$-theory is as shown in the table. For the complex part, we get (repeating what was obtained in \cite{MRV}) that
$KU_0( C \sp * \sr (\Lambda))$ is an extension of $\Z_g$ by $\Z_g$ and $KU_1( C \sp * \sr (\Lambda )) \cong \Z_g^2$.

It remains to show that $KU_0( C \sp * \sr (\Lambda)) \cong \Z_g^2$. 
We make use of the natural transformation 
$c \colon KO_i( C \sp * \sr (\Lambda)) \rightarrow KU_i( C \sp * \sr (\Lambda))$ which can be traced back from the spectral sequence as follows.

The map $c$ on $\mathcal{A}$ commutes with the chain maps $\partial_i$ and induces the map $c \colon (E^2_{p,q})\po \rightarrow (E^2_{p,q})\pu$. The map $c$ on each page of the spectral sequence induces the map $c$ on the following pages, and ultimately on the $E^\infty$ page. Finally the map $c \colon KO_i( C \sp * \sr (\Lambda)) \rightarrow KU_i( C \sp * \sr (\Lambda))$ commutes with the filtrations and on each subfactor is equal to the map $c$ obtained on the $E^\infty$ page. 

The real and complex parts of $\mathcal{A}$ are isomorphic in degree 0 and the complexification map $c_0$ on $\mathcal{A}$ actually implements this isomorphism, as seen in Figure~6. Furthermore, the maps $\partial_i$ are the same and commute with $c_0$. Thus $c$ is an isomorphism from the first row of the spectral sequence for  $KO_0( C \sp * \sr (\Lambda))$ to the first row of the spectral sequence for $KU_0( C \sp * \sr (\Lambda))$.

We now focus on the filtration of the $E^\infty$ page giving $KO_2( C \sp * \sr (\Lambda))$ and $KU_2( C \sp * \sr (\Lambda))$. Since $c$ commutes with the filtration we obtain the following diagram.
\[ \xymatrix{
0 \ar[r] &  (E^\infty_{0,2})\po \ar[r] \ar[d]^c
& KO_2(C \sp * \sr (\Lambda))  \ar[r] \ar[d]^c
&  (E^\infty_{2,0})\po  \ar[d]^c \ar[r]
& 0 \\
0 \ar[r] & (E^\infty_{0,2})\pu\ar[r] 
& KU_2(C \sp * \sr (\Lambda)) \ar[r] 
& (E^\infty_{2,0})\pu  \ar[r] 
& 0 \;  \\
} \;
\]
which can be rewritten as
\[ \xymatrix{
0 \ar[r] & 0 \ar[r] \ar[d]^c
& KO_2(C \sp * \sr (\Lambda)) \ar[r] \ar[d]^c
& \Z_g  \ar[d]^c \ar[r]
& 0 \\
0 \ar[r] & \Z_g \ar[r] 
& KU_2(C \sp * \sr (\Lambda)) \ar[r] 
& \Z_g  \ar[r] 
& 0 \;  \\
} \;
\]
Now since the vertical map $c$ on the right is an isomorphism, and the horizontal map
from $KO_2(  C \sp * \sr (\Lambda))$ is an isomorphism, the exact sequence on the bottom has a splitting. This proves that $KU_2 (C \sp * \sr (\Lambda)) \cong \Z_g^2$, and by periodicity $KU_i (C \sp * \sr (\Lambda)) \cong \Z_g^2$ for all even $i$.
\end{proof}

\begin{remark}
Compare this result with the calculations in \cite{Boersema02} of $K\crr(\mathcal{O}_n \pr)$ where $\mathcal{O}_n \pr$ is the real Cuntz algebra. We see that
the $\CR$-module $K\crr( C \sp * \sr (\Lambda))$ decomposes as a direct sum with four summands, each of which is isomorphic to 
$K\crr(\mathcal{O}_{g+1} \pr)$ or a certain suspension thereof.
Specifically,
$$K\crr( C \sp * \sr (\Lambda)) \cong K\crr(\mathcal{O}_{g+1} \pr) 
	\oplus (\Sigma^{-1} K\crr(\mathcal{O}_{g+1} \pr))^2 
		\oplus \Sigma^{-2} K\crr(\mathcal{O}_{g+1} \pr) \; .$$
\end{remark}

\begin{remark}
In the special case that $M_1 = T_1$, $M_2 = D_2$, $M_3 = D_3$,  the graph $\Lambda$ decomposes as a product graph (in the sense of Kumjian-Pask) of rank-1 graphs. Specifically, we have
$$\Lambda = \Lambda_1 \times \Lambda_2 \times \Lambda_3$$
where $\Lambda_1$ is a graph with two vertices and $2n_1$ edges from each vertex to the other;
and $\Lambda_2$, $\Lambda_3$ are graphs with 1 vertex and $2m_i$  loops.
Therefore,
$$C \sp *(\Lambda) = C \sp *(\Lambda_1) \otimes C \sp * (\Lambda_2) \otimes C \sp * (\Lambda_3) \; .$$
It can further be shown that all factors in this product are isomorphic to Cuntz algebras. Namely
$C \sp *(\Lambda_1) \cong \mathcal{O}_{4n_1^2 -1}$ and $C \sp * (\Lambda_i) \cong \mathcal{O}_{2m_i-1}$ (for $i = 2,3$).
Similarly, at the level of real $C \sp *$-algebras we have
$$C \sp * \sr (\Lambda) = C \sp * \sr (\Lambda_1) \otimes C \sp * \sr (\Lambda_2) \otimes C \sp * \sr (\Lambda_3) \; $$
where $C \sp * \sr (\Lambda_1) \cong \mathcal{O}\pr _{4n_1^2 -1}$ and $C \sp * \sr (\Lambda_i) \cong \mathcal{O}\pr_{2m_i-1}$ (for $i = 2,3$).

Therefore, this spectral sequence calculation above has given us an approach to calculating the $K$-theory of these products which is alternative to using the K\"unneth formula. The K\"unneth formula can be difficult when there are more than 2 factors and when there is torsion involved. This is especially true for the real case.

Also, in the more general case (without restriction on the forms of $M_i$), we find {\it a posteriori} (from Proposition~\ref{calc1}) that the $K$-theory depends only on the value of $g$. Using the classification theorems for purely infinite simple $C \sp *$-algebras (see the manuscripts of Kirchberg \cite{EK} and Phillips \cite{NCP} in the complex case and the work of the first author and collaborators in the real case \cite{BRS}) it follows that the isomorphism classes of $\Lambda$ depend only on the value of $g$. Therefore, in all cases the $C \sp *$-algebra $C \sp *(\Lambda)$ is isomorphic to an appropriate product of three Cuntz algebras (one with the same value of $g$) and similarly the real $C \sp *$-algebra $C \sp * \sr(\Lambda)$ is isomorphic to a product of three real Cuntz algebras.
\end{remark}

\begin{lemma} \label{snf1}
The Smith normal form of the matrices $\partial_1$, $\partial_2$, $\partial_3$ (in the complex part in degree 0) are equal to
\[ 
\snf(\partial_1) = \snf( \partial_3)^T =  \begin{bmatrix} 1 & 0 & 0 & 0 & 0 & 0 \\ 0 & g & 0 & 0 & 0 & 0 \end{bmatrix} \\
\quad \text{and} \quad
\snf(\partial_2) = \diag(1, 1, g, g, 0, 0)  \; .
\]
\end{lemma}

\begin{proof}
We note that the proof in Case (1) is correct in \cite{MRV}.

In Case (2), proceed as in the proof of Lemma 6.1 of \cite{MRV} where it is necessary to compute the Smith normal form of
$$\partial_1 = \begin{bmatrix} 1 & -2n_1 & 1 & -2n_2 & 1-2m_3 & 0 \\ -2n_1 & 1 & -2n_2 & 1 & 0 & 1 - 2m_3 \end{bmatrix} \; .$$
The list of the $2 \times 2$ minors (up to sign) of $\partial_1$ is
\begin{equation} \begin{gathered} 1 - 4n_1^2, ~1 - 4n_2^2, ~1 -4n_1 n_2, ~2(n_1 -n_2)  \\
 1-2m_3, ~(1 -2m_3)^2, ~2n_1(1-2m_3), ~2n_2(1-2m_3)
\end{gathered} \end{equation} 

As we are interested in the gcd of this list, we can clearly reduce everything on the second line to just $1-2m_3$.
Furthermore, by Lemma~\ref{number1} we can eliminate the last entry of the first row. Hence the $\gcd$ of the $2 \times 2$ minors is $g$ and
$\snf(\partial_1) = \diag(1,g)$. The result for $\partial_3$ is the same (up to transpose).

Now we consider $\partial_2$, where
$$\partial_2 = \begin{bmatrix} -1 &2n_2 & -(1-2m_3) & 0  & 0 & 0  \\ 2n_2 & -1 & 0 & -(1-2m_3) & 0 & 0 \\
1 & -2n_1 & 0 & 0 &  -(1-2m_3) & 0  \\ -2n_1 & 1 & 0 & 0 & 0 & -(1- 2m_3)  \\
0 & 0 & 1 & -2n_1 & 1 & -2n_2  \\ 0 & 0 & -2n_1 & 1 & -2n_2 & 1  
\end{bmatrix} \; $$
Here the $\gcd$ of the list of $2 \times 2$ minors is seen to be 1. Furthermore, each $4 \times 4$ minors is a product of at least 2 factors from the list of $2 \times 2$ minors. It follows that the $\gcd$ of the list of $4 \times 4$ minors is $g^2$ and that $\snf(\partial_2) = \diag(1,1,g,g,0,0)$.

For Case (3), we consider 
$$\partial_1 = \begin{bmatrix} 1 & -2n_1 & 1 & -2n_2 & 1 & -2n_3  \\ -2n_1 & 1 & -2n_2 & 1 & - 2n_3 & 1 \end{bmatrix} \; .$$
The list of the $2 \times 2$ minors is
\begin{equation} \begin{gathered} 1 - 4n_1^2, ~1 - 4n_2^2, ~ 1 - 4n_3^3, \\
~1 -4n_1 n_2, ~1 - 4n_1 n_3, ~1 - 4n_2 n_3 \\
~2(n_1 -n_2)  , ~2(n_1 - n_3), ~2 (n_2 - n_3) \; .
\end{gathered} \end{equation} 
Using Lemma~\ref{number3}, we can eliminate the third row of this set of formulae, giving the desired result.
This proves the result for $\partial_1$ and $\partial_3$. The calculation for $\partial_2$ is now similar to that in Case (2).

\end{proof}


\section{The Rank-3 case, with a non-trivial involution} \label{3-NT}

Now we consider the same rank-3 graph $\Lambda$ with a non-trivial involution $\gamma$. The involution $\gamma$ swaps the two vertices and this extends consistently to an involution on all higher-degree edges. The next theorem shows the $K$-theory of the real $C \sp *$-algebra $C \sp * \sr(\Lambda, \gamma)$. 

Again, we have three cases to consider, depending on the structure of the adjacency matrices.

\begin{enumerate}
\item If $M_1 = T_1, M_2 = D_2, M_3 = D_3$ then
\begin{align*}
g &= \gcd(1 - 4n_1^2, 1 - 2m_2, 1-2m_3) \\
h &= \gcd(1 - 2 n_1, 1 - 2 m_2,  1 - 2 m_3 ) \\
k &= \gcd(1 + 2 n_1, 1 - 2 m_2,  1 - 2 m_3 ) \; .
\end{align*}
\item If $M_1 = T_1, M_2 = T_2, M_3 = D_3$ then
\begin{align*}
g &= \gcd(1 - 4n_1^2, 1-4n_2^2, 1-4n_1 n_2, 1-2m_3)   \\
h &= \gcd(1 - 2 n_1, 1 - 2 n_2,  1 - 2 m_3 ) \\
k &= \gcd(1 + 2 n_1, 1 + 2n_2,  1 - 2 m_3 ) \; .
\end{align*}
\item If $M_1 = T_1, M_2 = T_2, M_3 = T_3$ then
\begin{align*}
g &= \gcd(1 - 4n_1^2, 1-4n_2^2, 1 - 4n_3^2,1-4n_1 n_2, 1-4n_1 n_3, 1-4n_2 n_3)   \\
h &= \gcd(1 - 2 n_1, 1 - 2 n_2,  1 - 2 n_3 ) \\
k &= \gcd(1 + 2 n_1, 1 + 2 n_2,  1 + 2 n_3 ) \; .
\end{align*}
\end{enumerate}
Note that $g = hk$ in each case, because $1 - 2n_i$ and $1 + 2n_i$ are relatively prime, by Lemmas~\ref{number2} and \ref{number4}.

\begin{prop} \label{rank-3-involution}
Let $\Lambda$ be a rank-3 graph as described above, with non-trivial involution $\gamma$. Then
$K\crr(C \sp * \sr (\Lambda, \gamma)$ is given by the table below.
\end{prop}

$$K\crr(C \sp *  \sr (\Lambda, \gamma))$$
$$\begin{array}{|c|c|c|c|c|c|c|c|c|}  
\hline \hline  
n & \makebox[1cm][c]{0} & \makebox[1cm][c]{1} & 
\makebox[1cm][c]{2} & \makebox[1cm][c]{3} 
& \makebox[1cm][c]{4} & \makebox[1cm][c]{5} 
& \makebox[1cm][c]{6} & \makebox[1cm][c]{7}  \\
\hline  \hline
KO_n(C \sp *  \sr (\Lambda, \gamma))
& \Z_h \oplus \Z_k &  \Z_h^2 & \Z_h \oplus \Z_k    & \Z_k^2
& \Z_h \oplus \Z_k &  \Z_h^2 & \Z_h \oplus \Z_k    & \Z_k^2 \\
\hline  
KU_n(C \sp *  \sr (\Lambda, \gamma))
& \Z_g^2 & \Z_g^2 & \Z_g^2 & \Z_g^2 & \Z_g^2 & \Z_g^2 & \Z_g^2 & \Z_g^2    \\
\hline \hline
\end{array}$$

\begin{remark}
Note that $\Z_g \cong \Z_h \oplus \Z_k$. 
So, similar to the previous examples, there is a direct sum decomposition. Here it can be written as
\begin{align*}
K\crr( C \sp * \sr (\Lambda, \gamma)) \cong &
 \left( K\crr(  \mathcal{O}_{h+1} \pr)  \oplus  \left( \Sigma^{-1} K\crr(\mathcal{O}_{h+1} \pr) \right)^2 \oplus \Sigma^{-2} K\crr(\mathcal{O}_{h+1} \pr) \right)
 \\
 & \oplus  \left( \Sigma^{-4} K\crr(  \mathcal{O}_{k+1} \pr)  \oplus \left( \Sigma^{-5} K\crr(\mathcal{O}_{k+1} \pr) \right)^2 \oplus \Sigma^{-6} K\crr(\mathcal{O}_{k+1} \pr) \right)
\; .
\end{align*} 
\end{remark} 

\begin{remark}
Let $A(n_1, m_2, m_3) = C \sp * \sr(\Lambda, \gamma)$ be the real $C \sp *$-algebra obtained from a particular choice of integers $n_1, m_2, m_3$ in Case (1). Let $\widetilde{g} = \gcd(m_2, m_3)$. If $n_1$ and $n'_1$ are two positive integers satifying $n_1 \equiv n_1' \pmod {\widetilde{g}}$
then we have
\begin{align*}
\gcd( 1 - 4n_1^2, 1 -2m_2, 1-2m_3) &= \gcd( 1 - 4(n_1')^2, 1 -2m_2, 1-2m_3) \\
\gcd( 1 - 2n_1, 1 -2m_2, 1-2m_3) &= \gcd( 1 - 2n_1', 1 -2m_2, 1-2m_3) \\
\gcd( 1 + n_1, 1 -2m_2, 1-2m_3) &= \gcd( 1 + 2n_1', 1 -2m_2, 1-2m_3)  \; . 
\end{align*}
Then it follows by Proposition~\ref{rank-3-involution} that 
$K\crr(A(n_1, m_2, m_3)) \cong K\crr( A(n_1', m_2, m_3))$ and therefore using \cite{BRS} that
$A(n_1, m_2, m_3)) \cong A(n_1', m_2, m_3))$.

Suppose on the other hand, we replace $n_1$ by $n'_1$ where $n_1 \equiv -n_1' \pmod {\widetilde{g}}$.
Then we have
\begin{align*}
\gcd( 1 - 4n_1^2, 1 -2m_2, 1-2m_3) &= \gcd( 1 - 4(n_1')^2, 1 -2m_2, 1-2m_3) \\
\gcd( 1 - 2n_1, 1 -2m_2, 1-2m_3) &= \gcd( 1 + 2n_1', 1 -2m_2, 1-2m_3) \\
\gcd( 1 + n_1, 1 -2m_2, 1-2m_3) &= \gcd( 1 - 2n_1', 1 -2m_2, 1-2m_3)  \; . 
\end{align*}
Thus the roles of the $h$ and $k$ change places and it follows by Proposition~\ref{rank-3-involution} that 
$K\crr(A(n_1, m_2, m_3)) \cong \Sigma^2 K\crr( A(n_1', m_2, m_3))$.
The real $C \sp *$-algebras $A(n_1, m_2, m_3)$ and $A(n_1', m_2, m_3)$ are not isomorphic in this case, though their respective complexifications are, since $KU_*(A(n_1, m_2, m_3)) \cong KU_*( A(n_1', m_2, m_3))$.
\end{remark} 

\begin{proof}[Proof of Proposition~\ref{rank-3-involution}]
The complex part $KU_*( C \sp * \sr (\Lambda, \gamma))$ is the same as what we obtained for $KU( C \sp * \sr (\Lambda))$, since both are isomorphic to the $K$-theory of the complex graph algebra associated to $\Lambda$, that is to $K_*( C \sp * (\Lambda))$. But to find $KO_*(C \sp * \sr (\Lambda, \gamma))$ we go back to the spectral sequence again.

By \cite{BG} there is again a chain complex
\[ 0 \rightarrow \mathcal{A} \xrightarrow{\partial_3} \mathcal{A}^3 \xrightarrow{\partial_2} 
				\mathcal{A}^3 \xrightarrow{\partial_1} \mathcal{A} \rightarrow 0 \]
the homology of which gives the $E^2$ page of a spectral sequence which converges to $K\crr( C \sp * (\Lambda))$, but this time we have
$\mathcal{A} = K\crr(\C)$.

The real part of this chain complex in even degrees is
\[ 0 \rightarrow \Z \xrightarrow{\partial_3} \Z^3 \xrightarrow{\partial_2}
				\Z^3 \xrightarrow{\partial_1} \Z \rightarrow 0 \]
and the real part vanishes in the odd degrees.
In degree 0 we have 
\begin{align*}
\partial_1 &= \begin{bmatrix} I - MO_1^T & I - MO_2^T & I - MO_3^T \end{bmatrix} \\
\partial_2 &= \begin{bmatrix} -(I - MO_2^T) & -(I - MO_3^T) & 0  \\
					I - MO_1^T & 0 & -(I - MO_3^T) \\
					0 & I - MO_1^T & I - MO_2^T &  \end{bmatrix} \\
\partial_3 &=\begin{bmatrix} I - MO_1^T \\ -(I - MO_2^T) \\ I - MO_3^T \end{bmatrix} 
\end{align*}
and the $1 \times 1$ matrices for $I - MO_i^T$ are found from $I - M_i^T$ using the instructions from Table~3 in Section~3D of \cite{BG}.

Now, we consider Case (1) specifically, so that we have
\[ I -   M_i = 
\begin{cases} 
\begin{bmatrix} 1 & -2 n_i \\ -2n_i & 1\end{bmatrix}  &i = 1 \\ ~ \\
\begin{bmatrix} 1- 2m_i & 0 \\ 0 & 1 - 2 m_i  \end{bmatrix}    & i = 2,3.                             
\end{cases} \]
We then find that
\[ I - MO_i = 
\begin{cases} 1 - 2n_i & i = 1 \\
 1 - 2m_i & i = 2,3
\end{cases} \]
(the rule here is that we add the entries in the first row of each $2 \times 2$ matrix to get a new $1 \times 1$ matrix).
So we have
\[ 
 \partial_1 = \partial_3^T =  \begin{bmatrix} 1 - 2n_1 & 1-2m_2 & 1 - 2m_3\end{bmatrix} 
\quad \text{and} \quad 
 \partial_2 = 
 	\begin{bmatrix} -(1 - 2m_2) & -(1 - 2m_3)  & 0 \\ 1 - 2n_1 & 0 & -(1 - 2m_3) \\ 0 &  1-2n_1 & 1 - 2m_2 \end{bmatrix} 
\]
We claim that
\[ 
S( \partial_1) = S( \partial_3^T) =  \begin{bmatrix} h & 0 & 0 \end{bmatrix} 
\quad \text{and} \quad 
S( \partial_2) = \begin{bmatrix} h & 0 & 0 \\ 0 & h & 0 \\ 0 & 0 & 0 \end{bmatrix} 
\]
where $h = \gcd(1 - 2n_1, 1 - 2m_2, 1 - 2m_3)$.
The statements about $S( \partial_1) $ and $S( \partial_3)$ are clear, but for $S(\partial_2)$ first note that $\partial_2$ has rank 2. The $\gcd$ of all the entries of $\partial_2$ is $h$; while the $\gcd$ of all the $2 \times 2$ minors 
\[ (1 - 2n_1)(1 - 2m_2), (1 - 2n_1)(1 - 2m_3), (1 - 2m_2)(1-m_3),  (1 - 2n_1)^2, (1 - 2m_2)^2, (1 - 2m_3)^2  \;  \]
which is $h^2$. From this the statement about $\snf(\partial_2)$ follows.

The result is that the homology of the chain complex in degree 0 is
$H_*(\mathcal{C}) = ( \Z_h, \Z_h^2, \Z_h, 0)$ in degrees $p = 0,1,2,3$ and this gives us the 0th row of the $E^2$ page of the real part of the spectral sequence. 

For row 2 of the spectral sequence, Table~3 of \cite{BG} dictates that we subtract instead of add the adjacent entries of $I - M_i$ so we have
\[ 
 \partial_1^T =  \begin{bmatrix} 1 + 2n_1 & 1-2m_2 & 1 - 2m_3\end{bmatrix} 
\quad \text{and} \quad 
 \partial_2 = 
 	\begin{bmatrix} -(1 - 2m_2) & 1 - 2m_3  & 0 \\ 1 + 2n_1 & 0 & -(1 - 2m_3) \\ 0 &  -(1+2n_1) & 1 - 2m_2 \end{bmatrix} 
\]
Thus 
\[ 
S( \partial_1) = S( \partial_3^T) =  \begin{bmatrix} k & 0 & 0 \end{bmatrix} 
\quad \text{and} \quad 
S( \partial_2) = \begin{bmatrix} k & 0 & 0 \\ 0 & k & 0 \\ 0 & 0 & 0 \end{bmatrix} 
\]
where $k = \gcd(1 + 2n_1, 1-2m_2, 1 - 2m_3)$.
The homology of the chain complex is 
$H_*(\mathcal{C}) = ( \Z_k, \Z_k^2, \Z_k, 0)$ in degrees $p = 0,1,2,3$ and this gives us row 2 of the $E^2$ page of the real part of the spectral sequence. 

Rows 4 and 6 are the same as rows 0 and 2, respectively. So the $E^2$ page of the spectral sequence is the following. For both the real and complex parts, we have $E^2 = E^\infty$, because no non-zero differentials are possible.
	\begin{gather*}
			\text{ \underline{ $E^2_{p,q} $ (for $g$ odd)  }}\\
			\def\vvline{\hfil\kern\arraycolsep\vline\kern-\arraycolsep\hfilneg}
			\begin{array}{ ccccc }
				\multicolumn
				{5}{c}{ \underline{ \text{real part}}} \\
				\vspace{.25cm} \\
				\vdots \vvline & \hspace{.3cm} \vdots \hspace{.3cm} & \hspace{.3cm} \vdots \hspace{.3cm} 
						& \hspace{.3cm} \vdots \hspace{.3cm} & \hspace{.1cm} \vdots \hspace{.1cm} \\
				7  \vvline & 0 & 0  & 0 & 0  \\
				6  \vvline  &  \Z_{k} &  \Z_k^2 & \Z_k  & 0 \\  
				5   \vvline & 0 & 0  & 0 & 0 \\
				4   \vvline  &  \Z_{h} &  \Z_h^2 & \Z_h & 0  \\
				3   \vvline & 0 & 0  & 0  & 0 \\
				2  \vvline  &  \Z_{k} &  \Z_k^2 & \Z_k  & 0 \\  
				1  \vvline & 0 & 0  & 0 & 0 \\ 
				0  \vvline  &  \Z_{h} &  \Z_h^2 & \Z_h  & 0 \\  \hline
				~ \vvline & 0 & 1 & 2 & 3
			\end{array}
			\hspace{3cm}
			\begin{array}{ ccccc }
				\multicolumn
				{5}{c}{ \underline{ \text{complex part}}} \\
				\vspace{.25cm} \\
				\vdots \vvline & \hspace{.3cm} \vdots \hspace{.3cm} & \hspace{.3cm} \vdots \hspace{.3cm} 
						& \hspace{.3cm} \vdots \hspace{.3cm} & \hspace{.1cm} \vdots \hspace{.1cm} \\
				7 \vvline  & 0    &  0 &  0  & 0 \\
				6    \vvline  &  \Z_{g} &  \Z_g^2 & \Z_g & 0  \\
				5  \vvline  & 0    &  0 &  0 & 0   \\
				4   \vvline  &  \Z_{g} &  \Z_g^2 & \Z_g & 0 \\
				3   \vvline  & 0    &  0 &  0  & 0  \\
				2    \vvline  &   \Z_{g} &  \Z_g^2 & \Z_g & 0\\
				1   \vvline  & 0    &  0 &  0  & 0  \\
				0   \vvline  &  \Z_{g} &  \Z_g^2 & \Z_g & 0 \\ \hline
				~ \vvline & 0 & 1 & 2 & 3
			\end{array}
		\end{gather*}
From the spectral sequence we immediately find the isomorphism class of
$KO_j( C \sp * \sr (\Lambda, \gamma))$ when $j$ is odd.
Now, for $j$ even there is a short exact sequence
$$0 \rightarrow \Z_h \rightarrow KO_0( C \sp * \sr (\Lambda, \gamma)) \rightarrow \Z_k \rightarrow 0 \; ,$$
or
$$0 \rightarrow \Z_k \rightarrow KO_0( C \sp * \sr (\Lambda, \gamma)) \rightarrow \Z_h \rightarrow 0 \; ,$$
depending on the parity of $j/2$. But since $h, k$ are relatively prime we must have 
$KO_0( C \sp * \sr (\Lambda, \gamma)) \cong \Z_h \oplus \Z_k \cong \Z_g$ in both cases.

The proofs in Cases (2) and (3) proceed similarly.
\end{proof}


\section{Rank-4 graph with 2 vertices} \label{4}

Now let $\Lambda$ be the rank-4 graph with 2 vertices discussed in Section 6 of \cite{MRV}. 
In Proposition~6.4 of \cite{MRV} some partial results are described for $K_*( C \sp * (\Lambda))$.
We again find that using the real and complex $K$-theory together, we can complete these computations.
In this section we present the $K$-theory of both real $C \sp *$-algebras $C \sp * \sr(\Lambda)$ and $C \sp * \sr(\Lambda, \gamma)$
where $\gamma$ is the non-trivial involution.
We also show the additional complications for $k > 4$ which prevent us from making further progress.

The adjacency matrices $M_i$ for $\Lambda$ are all of the form $T_i$ or $D_i$, as before.
We define $g$ as follows, according to the four possible cases. We also define $h$ and $k$ for reference when describing $KO_*(C \sp *  \sr(\Lambda, \gamma))$.

\begin{enumerate}
\item If $M_1 = T_1, M_2 = D_2, M_3 = D_3, M_4 = D_4$ then
\begin{align*}
	g &= \gcd\{1 - 4n_1^2, 1 - 2m_k \mid  k \in \{2, 3,4 \} \}  \\ 
	h &= \gcd\{1 - 2n_1, 1 - 2m_k \mid  k \in \{2, 3,4 \} \}  \\ 
	k &= \gcd\{1 + 2n_1, 1 - 2m_k \mid  k \in \{2, 3,4 \} \}  
\end{align*}
\item If $M_1 = T_1, M_2 = T_2, M_3 = D_3, M_4 = D_4$ then
\begin{align*}
	g &= \gcd\{1 - 4n_i^2,1-4n_i n_j, 1 -2m_k \mid i,j \in \{1,2\} , k \in \{3,4 \} \}  \\
	h &= \gcd\{1 - 2n_i , 1 -2m_k  \mid i \in \{1,2\} , k \in \{3,4 \} \}  \\
	k &= \gcd\{1 + 2n_i , 1 -2m_k \mid i \in \{1,2\}. , k \in \{3,4 \} \}  
	\; . 
\end{align*}
\item If $M_1 = T_1, M_2 = T_2, M_3 = T_3,  M_4 = D_4$ then
\begin{align*}
	g &= \gcd\{1 - 4n_i^2,1-4n_i n_j, 1 -2m_4 \mid i,j \in \{1,2,3\} \}  \\
	h &= \gcd\{1 - 2n_i , 1 -2m_4  \mid i \in \{1,2,3\}  \}  \\
	k &= \gcd\{1 + 2n_i , 1 -2m_4 \mid i \in \{1,2,3\} \}    		
	\; .
\end{align*}
\item If $M_1 = T_1, M_2 = T_2, M_3 = T_3,  M_4 = T_4$ then
\begin{align*}
	g &= \gcd\{1 - 4n_i^2,1-4n_i n_j  \mid i,j \in \{1,2,3,4\} \}  \\
	h &= \gcd\{1 - 2n_i  \mid i \in \{1,2,3,4\} \}  \\
	k &= \gcd\{1 + 2n_i  \mid i \in \{1,2,3,4\} \}    		\; .
\end{align*}
\end{enumerate}

Then using the methods of the two previous sections, we obtain the following propositions.

\begin{prop}
For a rank-4 graph as described above, with non-trivial involution $\gamma$ we have $K\crr( C \sp *( \Lambda))$ and $K\crr(C \sp * \sr (\Lambda, \gamma)$ given by the table below, for $g \geq 3$.
If $g = 1$, then $K\crr(C \sp * \sr(\Lambda)) = K\crr(C \sp * \sr (\Lambda, \gamma) = 0$.
\end{prop}


$$K\crr(C \sp * \sr(\Lambda))$$
$$\begin{array}{|c|c|c|c|c|c|c|c|c|}  
\hline \hline  
n & \makebox[1cm][c]{0} & \makebox[1cm][c]{1} & 
\makebox[1cm][c]{2} & \makebox[1cm][c]{3} 
& \makebox[1cm][c]{4} & \makebox[1cm][c]{5} 
& \makebox[1cm][c]{6} & \makebox[1cm][c]{7}  \\

\hline  \hline
KO_n(C \sp * \sr (\Lambda))
& \Z_g  & \Z_g^3   & \Z_g^3 & \Z_g 
& \Z_g  & \Z_g^3   & \Z_g^3 & \Z_g  \\
\hline  
KU_n(C \sp * \sr (\Lambda))
& \Z_g^4 & \Z_g^4 & \Z_g^4 & \Z_g^4 & \Z_g^4 & \Z_g^4 & \Z_g^4 & \Z_g^4    \\
\hline \hline
\end{array}$$


$$K\crr(C \sp * \sr(\Lambda, \gamma))$$
$$\begin{array}{|c|c|c|c|c|c|c|c|c|}  
\hline \hline  
n & \makebox[1cm][c]{0} & \makebox[1cm][c]{1} & 
\makebox[1cm][c]{2} & \makebox[1cm][c]{3} 
& \makebox[1cm][c]{4} & \makebox[1cm][c]{5} 
& \makebox[1cm][c]{6} & \makebox[1cm][c]{7}  \\

\hline  \hline
KO_n(C \sp * \sr (\Lambda, \gamma)) 
& \Z_h \oplus \Z_k^3 & \Z_h^3 \oplus \Z_k & \Z_h^3 \oplus \Z_k & \Z_h \oplus \Z_k^3
& \Z_h \oplus \Z_k^3 & \Z_h^3 \oplus \Z_k& \Z_h^3 \oplus \Z_k & \Z_h \oplus \Z_k^3 \\
\hline  
KU_n(C \sp * \sr (\Lambda, \gamma))
& \Z_g^4 & \Z_g^4 & \Z_g^4 & \Z_g^4 & \Z_g^4 & \Z_g^4 & \Z_g^4 & \Z_g^4    \\
\hline \hline
\end{array}$$

\begin{proof}
We first consider the spectral sequence for $K\crr( C \sp *  \sr (\Lambda) )$.
The incidence matrices $M_i$ for $\Lambda$ are of the form $D_i$ and $T_i$ as above, for $i = 1,2,3,4$. The chain complex we consider is
\[ 0 \rightarrow \mathcal{A} \xrightarrow{\partial_4} \mathcal{A}^4 \xrightarrow{\partial_3} \mathcal{A}^6 
				\xrightarrow{\partial_2} \mathcal{A}^4
				\xrightarrow{\partial_1} \mathcal{A} \rightarrow 0 \]
where $\mathcal{A} = K\crr(\R) \oplus K\crr(\R)$. The analysis of Propositions 6.3 and 6.4 of \cite{MRV} obtains the following chain complex in the degree 0 complex part,
	\[ 0 \rightarrow \Z^2 \xrightarrow{\partial_4} \Z^8 \xrightarrow{\partial_3} \Z^{12}
				\xrightarrow{\partial_2}  \Z^8
				\xrightarrow{\partial_1} \Z^2 \rightarrow 0  ; . \]
Now, following the method of calculation of \cite{MRV} but using the corrections as in Section~\ref{3-T} we find the following
\begin{align*}
\snf(\partial_1) = \snf( \partial_3)^T &=  \begin{bmatrix} 1 & 0 & \zero \\ 0 & g & \zero  \end{bmatrix}
			\text{~in~} M_{2, 8} (\R)  \\
 \text{and} \quad
\snf(\partial_2) = \snf( \partial_4)^T &= 
	\begin{bmatrix} I_3 & 0_3 & \zero \\ 0_3 & g I_3 & \zero \\ \zero & \zero & \zero \end{bmatrix} 
	\text{~in~} M_{8, 12} (\R)
\end{align*}
Thus in the even degree complex part we have $H_*(\mathcal{C}) = (\Z_g, \Z_g^3, \Z_g^3, \Z_g, 0)$.
The real part of the chain complex in degrees 0 and 4 are the same. The real part in degrees 1 and 2 are the same modulo 2, so $H_*(\mathcal{C}) = 0$ in those degrees.
The $E^2$ page of the spectral sequence in the real and complex parts are then given by
	\begin{gather*}
			\text{ \underline{ $E^2_{p,q} $ (for $g$ odd)  }}\\
			\def\vvline{\hfil\kern\arraycolsep\vline\kern-\arraycolsep\hfilneg}
			\begin{array}{ cccccc }
				\multicolumn
				{5}{c}{ \underline{ \text{real part}}} \\
				\vspace{.25cm} \\
				\vdots \vvline & \hspace{.3cm} \vdots \hspace{.3cm} & \hspace{.3cm} \vdots \hspace{.3cm} 
						& \hspace{.3cm} \vdots \hspace{.3cm}  & \hspace{.3cm} \vdots \hspace{.3cm} & \hspace{.1cm} \vdots \hspace{.1cm} \\
				7  \vvline & 0 & 0  & 0 & 0  & 0   \\
				6  \vvline & 0 & 0  & 0 & 0 & 0 \\
				5   \vvline & 0 & 0  & 0 & 0 & 0 \\
				4   \vvline  &  \Z_{g} &  \Z_g^3 & \Z_g^3 & \Z_g & 0  \\
				3   \vvline & 0 & 0  & 0  & 0  & 0 \\
				2   \vvline & 0 & 0  & 0  & 0 & 0 \\
				1  \vvline & 0 & 0  & 0 & 0 & 0 \\
				0  \vvline  &  \Z_{g} &  \Z_g^3 & \Z_g^3 & \Z_g  & 0 \\  \hline
				~ \vvline & 0 & 1 & 2 & 3 & 4
			\end{array}
			\hspace{3cm}
			\begin{array}{ cccccc }
				\multicolumn
				{5}{c}{ \underline{ \text{complex part}}} \\
				\vspace{.25cm} \\
				\vdots \vvline & \hspace{.3cm} \vdots \hspace{.3cm} & \hspace{.3cm} \vdots \hspace{.3cm} 
						& \hspace{.3cm} \vdots \hspace{.3cm}  & \hspace{.3cm} \vdots \hspace{.3cm} & \hspace{.1cm} \vdots \hspace{.1cm} \\							7 \vvline  & 0    &  0 &  0  & 0 & 0  \\
				6    \vvline  &  \Z_{g} &  \Z_g^3 & \Z_g ^3 & \Z_g & 0  \\
				5  \vvline  & 0    &  0 &  0 & 0 & 0   \\
				4   \vvline  &  \Z_{g} &  \Z_g^3 & \Z_g ^3 & \Z_g & 0  \\
				3   \vvline  & 0    &  0 &  0  & 0 & 0  \\
				2    \vvline  &   \Z_{g} &  \Z_g^3 & \Z_g ^3 & \Z_g & 0  \\
				1   \vvline  & 0    &  0 &  0  & 0 & 0  \\
				0   \vvline  &  \Z_{g} &  \Z_g^3 & \Z_g ^3 & \Z_g & 0  \\ \hline
				~ \vvline & 0 & 1 & 2 & 3 & 4
			\end{array}
		\end{gather*}

We note that again, the complexification map $c$ is an isomorphism on $E^2_{p, 0}$, the bottom row of the spectral sequence.
In the real part of this spectral sequence all differentials must vanish but in the complex part, while $d_2 = 0$ there appears to be possible non-zero differentials $d_3$. We use the map c to show that $d_3 = 0$.
The complexification map $c$ is an isomorphism in degree 0 on $\mathcal{A}$ and passes to a map $c$ which is an isomorphism in on the first row of the $E^2 = E^3$ pages of the spectral sequence. Furthermore, $c$ commutes with $d_3$. In particular there is a commutative diagram

\[ \xymatrix{
(E_{3, 0}^3) \po \ar[r]^{d_3} \ar[d]^c
& (E_{0, 2}^3) \po \ar[d]^c 
& \text{or}  & \Z_g \ar[r]^{d_3} \ar[d]^c 
& 0 \ar[d]^c \\
(E_{3, 0} ^3) \pu \ar[r]^{d_3} 
& (E_{0, 2}^3) \pu 
& & \Z_g \ar[r]^{d_3} 
& \Z_g  } \]
Since $(E_{0,2}^3) \po = 0$ and $c_{3,0}^3$ is an isomorphism, the commutative diagram forces
$d_3 \colon (E_{3,0}^3)\pu \rightarrow (E_{0,2}^3)\pu$ to vanish. By periodicity $d_3$ vanishes on all 
$(E_{3,i}^3)\pu$. Thus $E^2 = E^3 = E^\infty$ on both the real and complex parts.

Now, in the real part there are no extension problems so the calculation of $KO_*( C \sp * \sr (\Lambda))$ is complete. But in the complex part, there are questions of extensions for both $KU_0( C \sp * \sr (\Lambda))$ and $KU_1( C \sp * \sr (\Lambda))$. In both cases, we use  the complexification map $c$ as in the rank 3 case to show that the extension is split. First we use the $p + q = 2$ diagonal of the spectral sequence to find a diagram involving $c \colon KO_2(C \sp * \sr (\Lambda)) \rightarrow KU_2(C \sp * \sr (\Lambda))$.
\[ \xymatrix{
0 \ar[r] & 0 \ar[r] \ar[d]^c
& KO_2(C \sp * \sr (\Lambda)) \ar[r] \ar[d]^c
& \Z_g^3  \ar[d]^c \ar[r]
& 0 \\
0 \ar[r] & \Z_g \ar[r] 
& KU_2(C \sp * \sr (\Lambda)) \ar[r] 
& \Z_g^3  \ar[r] 
& 0 \\
}
\]
The vertical map $c$ on the right is an isomorphism, coming from the first row of the spectral sequence. This shows that the extension on the bottom of the diagram has a splitting and thus that $KU_2( C \sp * \sr (\Lambda)) \cong \Z_g^2$. 

Using the $p + q = 3$ diagonal, we obtain the diagram
\[ \xymatrix{
0 \ar[r] & 0 \ar[r] \ar[d]^c
& KO_3(C \sp * \sr (\Lambda)) \ar[r] \ar[d]^c
& \Z_g  \ar[d]^c \ar[r]
& 0 \\
0 \ar[r] & \Z_g^3 \ar[r] 
& KU_3(C \sp * \sr (\Lambda)) \ar[r] 
& \Z_g  \ar[r] 
& 0 \\
}
\]
where again the vertical map $c$ is an isomorphism and again we find that $KU_3(C \sp * \sr (\Lambda)) = \Z_g^2$. 
Therefore $KU_i(C \sp * \sr (\Lambda)) = \Z_g^2$ for all $i$. This completes the calculation of $K\crr( C \sp * \sr(\Lambda))$.

For $K\crr( C \sp * \sr (\Lambda, \gamma))$,
the $E^2$ page of the spectral sequence in the real and complex parts are as follows.
	\begin{gather*}
			\text{ \underline{ $E^2_{p,q} $ (for $g$ odd)  }}\\
			\def\vvline{\hfil\kern\arraycolsep\vline\kern-\arraycolsep\hfilneg}
			\begin{array}{ cccccc }
				\multicolumn
				{5}{c}{ \underline{ \text{real part}}} \\
				\vspace{.25cm} \\
				\vdots \vvline & \hspace{.3cm} \vdots \hspace{.3cm} & \hspace{.3cm} \vdots \hspace{.3cm} 
						& \hspace{.3cm} \vdots \hspace{.3cm}  & \hspace{.3cm} \vdots \hspace{.3cm} & \hspace{.1cm} \vdots \hspace{.1cm} \\
				7  \vvline & 0 & 0  & 0 & 0  & 0   \\
				6   \vvline  &  \Z_{k} &  \Z_k^3 & \Z_k^3 & \Z_k  & 0 \\
				5   \vvline & 0 & 0  & 0 & 0 & 0 \\
				4   \vvline  &  \Z_{h} &  \Z_h^3 & \Z_h^3 & \Z_h  & 0  \\
				3   \vvline & 0 & 0  & 0  & 0  & 0 \\
				2   \vvline  &  \Z_{k} &  \Z_k^3 & \Z_k^3 & \Z_k  & 0 \\
				1  \vvline & 0 & 0  & 0 & 0 & 0 \\
				0  \vvline  &  \Z_{h} &  \Z_h^3 & \Z_h^3 & \Z_h  & 0 \\  \hline
				~ \vvline & 0 & 1 & 2 & 3 & 4
			\end{array}
			\hspace{3cm}
			\begin{array}{ cccccc }
				\multicolumn
				{5}{c}{ \underline{ \text{complex part}}} \\
				\vspace{.25cm} \\
				\vdots \vvline & \hspace{.3cm} \vdots \hspace{.3cm} & \hspace{.3cm} \vdots \hspace{.3cm} 
						& \hspace{.3cm} \vdots \hspace{.3cm}  & \hspace{.3cm} \vdots \hspace{.3cm} & \hspace{.1cm} \vdots \hspace{.1cm} \\
				7 \vvline  & 0    &  0 &  0  & 0 & 0  \\
				6    \vvline  &  \Z_{g} &  \Z_g^3 & \Z_g ^3 & \Z_g & 0  \\
				5  \vvline  & 0    &  0 &  0 & 0 & 0   \\
				4   \vvline  &  \Z_{g} &  \Z_g^3 & \Z_g ^3 & \Z_g & 0  \\
				3   \vvline  & 0    &  0 &  0  & 0 & 0  \\
				2    \vvline  &   \Z_{g} &  \Z_g^3 & \Z_g ^3 & \Z_g & 0  \\
				1   \vvline  & 0    &  0 &  0  & 0 & 0  \\
				0   \vvline  &  \Z_{g} &  \Z_g^3 & \Z_g ^3 & \Z_g & 0  \\ \hline
				~ \vvline & 0 & 1 & 2 & 3 & 4
			\end{array}
		\end{gather*}
The complex part of this is the same as what we analyzed in the first part of this proof.
For the real part, since $h$ and $k$ are relatively prime, we have $d_r = 0$ for all $r$, so $(E^2_{p,q})\po = (E^\infty_{p,q})\po$. Furthermore, along each diagonal $p + q = n$, the extensions determining $KO_n( C \sp * \sr (\Lambda, \gamma))$ must be direct sums, again because $h$ and $k$ are relatively prime.
\end{proof}

\begin{remark}
The $\CR$-modules $K\crr( C \sp * \sr (\Lambda))$ and $K\crr(C \sp * \sr(\Lambda, \gamma)$ can be seen to have decompositions  as a direct sum with eight summands, 
$$K\crr( C \sp * \sr (\Lambda)) \cong 
	K\crr(\mathcal{O}_{g+1} \pr) 
		\oplus (\Sigma^{-1} K\crr(\mathcal{O}_{g+1} \pr) )^3
			\oplus (\Sigma^{-2} K\crr(\mathcal{O}_{g+1} \pr) )^3
				\oplus \Sigma^{-3} K\crr(\mathcal{O}_{g+1} \pr) \; $$
and
				\begin{align*}
K\crr( C \sp * \sr (\Lambda, \gamma)) \cong &
 \left( K\crr(\mathcal{O}_{h+1} \pr) 
		\oplus (\Sigma^{-1} K\crr(\mathcal{O}_{h+1} \pr) )^3
			\oplus (\Sigma^{-2} K\crr(\mathcal{O}_{h+1} \pr) )^3
				\oplus \Sigma^{-3} K\crr(\mathcal{O}_{h+1} \pr)  \right)
 \\
 & \oplus  \left(  K\crr(\mathcal{O}_{k+1} \pr) 
		\oplus (\Sigma^{-1} K\crr(\mathcal{O}_{k+1} \pr) )^3
			\oplus (\Sigma^{-2} K\crr(\mathcal{O}_{k+1} \pr) )^3
				\oplus \Sigma^{-3} K\crr(\mathcal{O}_{k+1} \pr)  \right)
\; .
\end{align*} 
\end{remark}

\begin{remark}
We note that if $k \geq 5$, there will be an extra non-zero column to these spectral sequences. We can still analyze the spectral sequence and find that $(E^2_{p,q})\po = (E^\infty_{p,q})\po$ and $(E^2_{p,q})\pu = (E^\infty_{p,q})\pu$, using similar arguments as in the previous cases. However, there will be extension problems that we are unable to determine. For example,
$KO_0( C \sp * \sr (\Lambda))$ will be an extension of $\Z_g$ by $\Z_g$ with no clear way to determine the isomorphism class of extension group.

Furthermore, if $k \geq 6$, there is a more fundamental problem. There will be two extra non-zero columns of the spectral sequences used in these computations. As a result, there will be the possibility of a non-zero differential, say $d_5  \colon E^5_{5,0} \rightarrow E^5_{0,4}$, in both the real and complex parts. We have no clear way of determining the value of this differential. In general, we have no understanding of how the differential maps $d_r$ relate to the structure of the higher rank graph.
\end{remark}

\section{Appendix: Number Theory Lemmas} \label{NTL}

\begin{lemma} \label{number1}
Let $n_1$ and $n_2$ be positive integers, greater than 1. Then $g_1 = g_2 = g_3$ where
\begin{align*}
g_1 &= \gcd(1 - 4n_1^2, 1 -4 n_2^2, 1 -4n_1 n_2) \\
g_2	&= \gcd(1 - 4n_1^2, 1 -4 n_2^2, 2n_1-2n_2 ) \\
g_3 &= \gcd(1 - 4n_1^2, 1 -4 n_2^2, 1 -4n_1 n_2, 2n_1-2n_2 )
\end{align*}
\end{lemma}
\begin{proof}
Suppose that $p^k$ is a odd prime power such that $p^k | \gcd( 1 - 4n_1^2, 1 - 4n_2^2)$. Since $1 - 4n_1^2 = (1- 2n_1)(1 + 2n_1)$, and since $1 - 2n_1$ and $1 + 2n_1$ are relatively prime, we have either $p^k | (1-2n_1)$ or $p^k | (1 + 2n_1)$. If $p^k | (1 - 2n_1)$ then $p^k$ divides
$$(1 - 2n_1)(1 + 2n_2) = (1 - 4n_1 n_2) - 2(n_1 - n_2) \; .$$
It follows that if $p^k$ divides one of $1 - 4n_1 n_2$ and $2(n_1-n_2)$, then it divides both.
Similarly, if $p^k | (1 + 2n_1)$ we find that $p^k$ divides
$$(1 + 2n_1)(1 - 2n_2) = (1 - 4n_1 n_2) + 2(n_1 - n_2)$$
and the same conclusion is made. This proves the lemma.
\end{proof}

\begin{lemma} \label{number2}
Let $n_1$ and $n_2$ be positive integers, greater than 1. Then $\gcd(h,k) = 1$ and $g = hk$ where
\begin{align*}
g &= \gcd(1 - 4n_1^2, 1 -4 n_2^2, 1 -4n_1 n_2) \\
h&= \gcd(1 - 2n_1, 1 - 2 n_2) \\
k&= \gcd(1 + 2n_1,  1 + 2 n_2)  \; .
\end{align*}
\end{lemma}

\begin{proof}
The first statement follows since $1 - 2n_1$ and $1 + 2n_1$ are relatively prime.

Let $p^\ell$ be a prime power and if $p^\ell| hk$, then $p^\ell | h$ or $p^\ell | k$.
As $1 - 4n_i^2 = (1 + 2n_i)(1 - 2n_i)$, it follows that $p^\ell | (1 - 4n_i^2)$ for both $i$. Furthermore, working modulo $p^\ell$, if
$p^\ell | h$ we have $2 n_1 \equiv 2 n_2 \equiv 1$ so $4n_1 n_2 \equiv 1$. If $p^\ell | k$ we have
$2n_1 \equiv 2n_2 \equiv -1$ so also $4n_1 n_2 \equiv 1$. Either way, $p^\ell | (1 - 4n_1 n_2)$, which implies $p^\ell | g$.

Conversely, suppose that $p^\ell | g$. So $p^\ell | (1 - 2n_i)$ or $p^\ell | (1 + 2n_i)$,  for each $i$.
If $p^\ell$ divides both $1 - 2n_1$ and $1 - 2n_2$, then by subtracting we find that $p^\ell | 2(n_1-n_2)$.
Similarly, if $p^\ell$ divides both $1 + 2n_1$ and $1 + 2n_2$, we also find $p^\ell | 2(n_1-n_2)$. 
In either of these cases we have $p^\ell | g$, using Lemma~\ref{number1}.

If on the other hand $p^\ell$ divides both $1 - 2n_1$ and $1 + 2n_2$, then modulo $p^\ell$ we have
$$1 \equiv 4n_1 n_2 \equiv (2 n_1)(2n_2) \equiv (1)(-1)  \equiv -1 \; .$$
This is a contradiction, as $p$ is odd.
\end{proof}

Using the same methods, we obtain the following extensions to these results.

\begin{lemma} \label{number3}
Let $n_1, \dots, n_\ell$ be positive integers, greater than 1. Then $g_1 = g_2 = g_3$ where
\begin{align*}
g_1 &= \gcd \{ 1 - 4n_i^2, 1 -4n_i n_j  \mid  i,j \in \{1, \dots \ell \}  \} \\
g_2 &= \gcd \{ 1 - 4n_i^2, 2n_i - 2n_j  \mid  i,j \in \{1, \dots \ell \}  \} \\
g_3 &= \gcd \{ 1 - 4n_i^2, 1 - 4n_i n_j,  2n_i - 2n_j  \mid  i,j \in \{1, \dots \ell \}  \} 
\end{align*}
\end{lemma}

\begin{lemma} \label{number4}
Let $n_1, \dots, n_\ell$ be positive integers, greater than 1. Then $\gcd(h,k) = 1$ and $g = hk$ where
\begin{align*}
g &= \gcd \{ 1 - 4n_i^2, 1 -4n_i n_j  \mid  i,j \in \{1, \dots \ell \}  \} \\
h&= \gcd \{ 1 - 2n_i   \mid  i \in \{1, \dots \ell \}  \} \\
k&= \gcd \{ 1 + 2n_i \mid  i \in \{1, \dots \ell \}  \} \\ \; 
\end{align*}
\end{lemma}

\end{document}